\documentclass[12pt]{article}
\usepackage[utf8]{inputenc}

\usepackage[left=2cm, right=2cm, top=2.50cm, bottom=2.50cm]{geometry}

\usepackage{amssymb}
\usepackage[all,arc]{xy}
\usepackage{mathrsfs}
\usepackage{extarrows}
\usepackage{xcolor}
\usepackage{amsmath}
\usepackage{amsfonts,epsfig}
\usepackage{amssymb,bm}
\usepackage{graphicx}
\usepackage{subfig}
\usepackage{pgf,pgfarrows,pgfnodes,pgfautomata,pgfheaps}
\usepackage{appendix}

\usepackage{tabularx}  
\usepackage{booktabs}
\usepackage{multirow}
\usepackage{array,makecell}

\usepackage{algorithm}
\usepackage{algpseudocode}
\usepackage{pifont}

\numberwithin{equation}{section}

\usepackage{amsmath,amsthm,amsfonts,amssymb,amscd,color}
\usepackage{hyperref}
\hypersetup{hypertex=true,
colorlinks=true,
linkcolor=blue,
anchorcolor=blue,
citecolor=blue}

\allowdisplaybreaks[4]
\newtheorem{theo}{Theorem}[section]
\newtheorem{lem}[theo]{Lemma}

\newtheorem{defi}[theo]{Definition}

\begin{document}

\newcommand{\mG}{\mathcal{G}}
\newcommand{\xu}{x_{u^*}}
\newcommand{\ga}{\theta_{1,2,3}}
\newcommand{\gb}{\theta_{1,2,4}}

\title{The maximum spectral radius of graphs of given size with forbidden subgraph\thanks{This work is partially supported by the National Natural Science Foundation of China (Grant No. 11971180), the Guangdong Provincial Natural Science Foundation (Grant No. 2019A1515012052).}}
\author{Xiaona Fang, Lihua You\thanks{Corresponding author: ylhua@scnu.edu.cn}\\
	{\small School of Mathematical Sciences, South China Normal University,}\\ {\small Guangzhou, 510631, P. R. China}
}
\date{}
\maketitle

{\bf Abstract}: Let $G$ be a graph of size $m$ and $\rho(G)$ be the spectral radius of its adjacency matrix. A graph is said to be $F$-free if it does not contain a subgraph isomorphic to $F$. In this paper, we prove that if $G$ is a $K_{2,r+1}$-free non-star graph with $m\geq (4r+2)^2+1$, then $\rho(G)\leq \rho(S_m^1)$, with equality if and only if $G\cong S_m^1$. Recently, Li, Sun and Wei \cite{li2022} showed that for any $\ga$-free graph of size $m\geq 8$, $\rho(G)\leq \frac{1+\sqrt{4m-3}}{2}$, with equality if and only if $G\cong S_{\frac{m+3}{2},2}$. However, this bound is not attainable when $m$ is even. We proved that if $G$ is $\ga$-free and $G\ncong S_{\frac{m+3}{2},2}$ with $m\geq 22$, then $\rho(G)\leq \rho(F_{m,1})$ if $m$ is even, with equality if and only if $G\cong F_{m,1}$, and $\rho(G)\leq \rho(F_{m,2})$ if $m$ is odd, with equality if and only if $G\cong F_{m,2}$.



\baselineskip=0.30in

\section{Introduction}\label{sec1}
\hspace{1.5em}
Let $A(G)$ be the adjacency matrix of a graph $G$ and $\rho(G)$ be the spectral radius of $A(G)$. A graph is said to be $F$-free if it does not contain a subgraph isomorphic to $F$. Maximizing the number of edges over all $F$-free graphs with $n$ vertices is a classical problem in graph theory. Basing on this, Nikiforov \cite{niki2010pn} proposed a spectral Tur\'an problem which asks to determine the maximum spectral radius of an $F$-free graph of order $n$. This is the spectral analogue of Tur\'an type extremal problem which was well studied in the literatures (see \cite{babaiKst,chen2019,taitminor,wilf1986Kr,zhao2021,zhai2020,zhai2022,zhai2012} and so on). For more results, readers are referred to a survey by Nikiforov \cite{niki2011}.

If we fix the number of edges instead of the number of vertices in the spectral Tur\'an problem, then we have the following question: what is the maximum spectral radius of an $F$-free graph with $m$ edges? Nosal \cite{nosal} proved that $\rho(G)\leq \sqrt{m}$ for every triangle-free graph $G$ of size $m$. When $G$ is non-bipartite and triangle-free, Lin, Ning and Wu \cite{lin2021} slightly improved the bound to $\rho(G)\leq \sqrt{m-1}$. In order to get a sharp spectral condition for $m>5$, Zhai and Shu \cite{zhai2022} further improved the bound to $\rho(G)\leq \rho^*(m)$ where $\rho^*(m)$ is the largest root of $x^3-x^2-(m-2)x+m-3=0$ for every triangle-free and non-bipartite graph $G$ of size $m$. Nikiforov \cite{niki2009c4} showed that if $G$ is $C_4$-free then $\rho(G)\leq \sqrt{m}$. It implies that $G$ contains $C_3$ and $C_4$ if $\rho(G)>\sqrt{m}$ combining with Nosal's result. Zhai, Lin and Shu \cite{zhai2021} followed this direction in determining which subgraphs will be contained in $G$ if $\rho(G)\geq f(m)$, where $f(m)\sim \sqrt{m}$ as $m \rightarrow \infty$. Min, Lou and Huang \cite{min2022} gave a sharp upper bound of the spectral radius of $C_5$-free or $C_6$-free graphs with given even size.

For a graph $F$ and an integer $m$, let $\mG (m, F)$ be the set of $F$-free graphs of size $m$ without isolated vertices. For $F=K_{r+1}$ or $K_{2,r+1}$, the following results are obtained in \cite{niki2002,niki2006,zhai2021}.

\begin{theo}{\rm (\cite{niki2002,niki2006})}\label{Kr+1}
	Let $G\in \mG(m,K_{r+1})$. Then $\rho(G)\leq \sqrt{2m(1-\frac{1}{r})}$, with equality if and only if $G$ is a complete bipartite graph for $r=2$, and $G$ is a complete regular $r$-partite graph for $r\geq 3$.
\end{theo}

\begin{theo}{\rm (\cite{zhai2021})}\label{zhaik2,r+1}
	Let $G$ be a graph of size $m$. If $G\in \mG (m,  K_{2,r+1})$ with $r\geq 2$ and $m\geq 16r^2$. Then $\rho(G)\leq \sqrt{m}$, and equality holds if and only if $G$ is a star.
\end{theo}

Let $S_n^k$ be the star of order $n$ with $k$ disjoint edges within its independent set. Following Theorem \ref{zhaik2,r+1}, we consider $K_{2,r+1}$-free graphs among all non-star graphs and obtain the following sharp result.

\begin{theo}\label{k2,r+1}
	Let $G\in \mG (m, K_{2,r+1})\setminus\{S_{m+1}\}$ with $r\geq 2$ and $m\geq (4r+2)^2+1$. Then $\rho(G)\leq \rho(S_m^1)$, where $\rho(S_m^1)$ is the largest real root of the equation $x^3-x^2-(m-1)x+m-3=0$, and equality holds if and only if $G\cong S_m^1$.
\end{theo}

For any three positive integers $p$, $q$, $r$, with $p\leq q \leq r$ and $q\geq 2$, let $\theta_{p,q,r}$ denote the graph obtained from three internally disjoint paths with the same pair of endpoints, where the three paths are of lengths $p$, $q$, $r$, respectively. We now consider graphs which are $\ga$-free. To state our results, we need some symbols for given graphs. For $1\leq k \leq n$, let $S_{n,k}$ be the graph of order $n$ obtained by joining each vertex of the complete graph $K_k$ to $n-k$ isolated vertices, i.e., $S_{n,k}=K_k\triangledown \overline{K}_{n-k}$. Let $F_{m,t}$ be the graph of size $m$ obtained by joining a vertex with maximum degree of $S_{\frac{m-t+3}{2},2}$ to $t$ isolated vertices (see Figure \ref{Fmt}). Li, Sun and Wei \cite{li2022} consider the case for $F=\ga$.




\begin{theo}{\rm (\cite{li2022})}\label{ga}
	Let $G$ be a graph in $\mG(m,\ga)$ where $m\geq 8$. Then $\rho(G)\leq \frac{1+\sqrt{4m-3}}{2}$, and equality holds if and only if $G\cong S_{\frac{m+3}{2},2}$.
\end{theo}

Considering that the extremal graph $S_{\frac{m+3}{2},2}$ has odd size $m$, for even $m$, the upper bound is not sharp. We study the extremal graph attaining the maximum spectral radius among $\mG(m,\ga)\setminus \{S_{\frac{m+3}{2},2}\}$, and obtain the following sharp results.

\begin{theo}\label{c5}
	Let $G\in \mG(m,\ga)\setminus \{S_{\frac{m+3}{2},2}\}$ with $m\geq 22$. Then $\rho(G)\leq \left\{\begin{matrix}
	\rho(F_{m,1}),	&\text{if } m \text{ is even}, \\ 
	\rho(F_{m,2}),	&\text{if } m \text{ is odd},
	\end{matrix}\right.$	
	where $\rho(F_{m,t})$ is the largest root of $x^4-mx^2-(m-t-1)x+\frac{t}{2}\cdot (m-t-1)=0$ for $t=1,2$. Moreover, the equality holds if and only if $G\cong \left\{\begin{matrix}
	F_{m,1},	&\text{if } m \text{ is even}, \\ 
	F_{m,2},	&\text{if } m \text{ is odd}.
	\end{matrix}\right.$
\end{theo}


%

\section{Preliminaries}
\hspace{1.5em}Let $G$ be a simple graph with vertex set $V(G)$ and edge set $E(G)$. For subsets $S$, $T$ of $V(G)$, we write $E(S,T)$ for the set of edges with one endpoint in $S$ and the other in $T$. Let $e(S,T)=|E(S,T)|$, and $e(S,S)$ is simplified by $e(S)$. For a vertex $v\in V(G)$, let $N(v)$ be the neighborhood of $v$ in $G$ and $N[v]=N(v)\cup \{v\}$. In particular, let $N_S(v)=N(v)\cap S$ and $d_S(v)=|N_S(v)|$ for $S\subset V(G)$. In the following, we introduce some basic lemmas. 

\begin{lem}{\rm \cite{niki2009c4}}\label{moveedge}
	Let $A$ and $A'$ be the adjacency matrices of two connected graphs $G$ and $G'$ on the same vertex set. Suppose that $N_G(u) \subsetneqq N_{G'}(u)$ for some vertex $u$. If the Perron vector $X$ of $G$ satisfies $X^TA'X\geq X^TAX$, then $\rho(G')>\rho(G)$.
\end{lem}

\begin{defi}{\rm (\cite{cve2010})}
	Given a graph $G$, the vertex partition $\Pi$: $V(G)=V_1\cup V_2 \cup \dots \cup V_k$ is said to be an equitable partition if, for each $u \in V_i$, $|V_j \cap N(u)|=b_{ij}$ is a constant depending only on $i, j\ (1 \leq i , j \leq k)$. The matrix $B_{\Pi}=(b_{ij})$ is called the quotient matrix of $G$ with respect to $\Pi$.
\end{defi}
\begin{lem}{\rm (\cite{cve2010})}\label{lemquo}
	Let $\Pi$: $V(G)=V_1\cup V_2 \cup \dots \cup V_k$ be an equitable partition of $G$ with quotient matrix $B_{\Pi}$. Then $det(xI-B_{\Pi})\mid det(xI-A(G))$. Furthermore, the largest eigenvalue of $B_{\Pi}$ is just the spectral radius of $G$.
\end{lem}

\begin{lem}{\rm (\cite{niki2009c4})}\label{lemsnk}
	Let $S_n^k$ be the star of order $n$ with $k$ disjoint edges within its independent set. Then $\rho(S_n^k)$ is the largest root of the equation $x^3-x^2-(n-1)x+n-1-2k=0$.
	
\end{lem}


\begin{lem}{\rm \cite{niki2010pn}}\label{snk}
	Let $S_{n,k}$ be the graph defined in Section \ref{sec1}. Then $\rho(S_{n,k})$ is the largest root of the equation $x^2-(k-1)x-k(n-k)=0$.\\
\end{lem}

\begin{lem}{\rm \cite{sun2020}}\label{-v}
	Let $G$ be a graph and let $v$ be a vertex in $V(G)$ with $d_G(v)\geq 1$. Then $\rho(G)\leq \sqrt{\rho^2(G-v)+2d_G(v)-1}$. Equality holds if and only if either $G\cong K_n$ or $G\cong K_{1,n-1}$ with $d_G(v)=1$.
\end{lem}

%
%
%


\section{Proof of Theorem \ref{k2,r+1}}
\hspace{1.5em}In this section, we consider $K_{2,r+1}$-free graphs among all non-star graphs. A graph is called $2$-connected, if it is a connected graph without cut vertex. 
Motivated by Lemma 2.2 of \cite{zhai2021}, we have the following Lemma.

\begin{lem}\label{lemcut}
	Let $m\geq 3$, $G^*$ be the extremal graph with the maximal spectral radius in $\mG(m,K_{2,r+1})\setminus \{S_{m+1}\}$, and $X^*=(x_{v_1},\dots,x_{v_{|V(G^*)|}})^T$ be the Perron vector of $G^*$ with $x_{u^*}=\max\limits_{u\in V(G^*)} x_u$. Then the following claims hold.\\
	{\rm (i)} $G^*$ is connected.\\
	{\rm (ii)} There exists no cut vertex in $V(G^*)\setminus \{u^*\}$, and hence $d(u)\geq 2$ for any $u\in V(G^*)\setminus N[u^*]$.
\end{lem}
\begin{proof}
	
	(i) Suppose that $G^*$ is disconnected. Let $G_1$ be a connected component of $G^*$ with $\rho(G^*)=\rho(G_1)$ and $G_2=G^*-G_1$. Selecting an edge $v_1v_2\in E(G_2)$ and a vertex $v$ with minimum degree in $G_1$. Let $G'$ be a graph obtained from $G^*-v_1v_2+v_1v$ by deleting all of its isolated vertices. Note that $K_{2,r+1}$ is $2$-connected and $v_1v$ is a cut edge in $G'$. Then $G'\in \mG(m,K_{2,r+1})$. Clearly, $G'\ncong S_{m+1}$ and $G_1$ is a proper subgraph of $G'$. Hence, $\rho(G')>\rho(G_1)=\rho(G^*)$, a contradiction.
	
	(ii) Suppose that there exists a cut vertex in $V(G^*)\setminus \{u^*\}$. Then $G^*$ has at least two end-blocks. We may assume that $B$ is an end-block of $G^*$ with $u^*\notin V(B)$ and a cut vertex $u\in V(B)$. Let $G''=G^*+\{u^*w \mid w\in N(u)\cap V(B) \}-\{uw \mid w\in N(u)\cap V(B) \}$. Then $G''\in \mG(m,K_{2,r+1})$ by $K_{2,r+1}$ is $2$-connected and $u^*$ is a cut vertex in $G''$. Since $N_{G^*}(u^*) \subsetneqq N_{G''}(u^*)$ and ${X^*}^T(A(G'')-A(G^*))X^*=\sum\limits_{u_iu_j\in E(G'')}2x_{u_i}x_{u_j}-\sum\limits_{u_iu_j\in E(G^*)}2x_{u_i}x_{u_j}=\sum\limits_{w\in N(u)\cap V(B)}2(x_{u^*}-x_{u})x_w\geq 0$, we have $\rho(G'')>\rho(G^*)$ by Lemma \ref{moveedge}. 
	
	Now we show that $G''\ncong S_{m+1}$. Denote by $D_{i,j}$ the double star which obtained from stars $S_{i+1}$ and $S_{j+1}$ by joining their centers, where $v_i$ is the center of $S_{i+1}$, $v_j$ is the center of $S_{j+1}$, and $K$ is the set of any $j-1$ terminal vertices of $S_{j+1}$.
	We only need to show $G^*\ncong D_{i,m-1-i}$ for $1\leq i \leq m-2$ by the construction of $G''$. 
	
	 Without loss of generality, we supposed that $x_i\geq x_j$. Clearly, $D_{m-2,1}=D_{i,j}+\{v_iw\mid w\in K\}-\{v_jw\mid w\in K\}$ where $i+j=m-1$. By Lemma \ref{moveedge}, we have $\rho(D_{i,m-1-i})\leq \rho(D_{m-2,1})$ for $1\leq i \leq m-2$. 
	
	Since $S_m^1\in \mG(m,K_{2,r+1})\setminus \{S_{m+1}\}$, we have $\rho(G^*)\geq \rho(S_m^1)$. By Lemma \ref{lemsnk}, $\rho(S_m^1)$ is the largest root of $f(x)=0$ where $f(x)=x^3-x^2-(m-1)x+m-3$. By a direct calculation, $\rho(D_{m-2,1})$ is the largest root of $g(x)=0$ where $g(x)=x^4-mx^2+m-2$. Let $h(x)=g(x)-xf(x)=x^3-x^2-(m-3)x+m-2$. then $h'(x)=3x^2-2x-(m-3)>0$ when $x>\sqrt{m-1}$, which means $h(x)$ is increasing in interval $(\sqrt{m-1},\infty)$. Moreover, $h(\sqrt{m-1})=2\sqrt{m-1}-1>0$. Thus, $h(x)>0$ when $x>\sqrt{m-1}$. Since $\rho(S_m^1)>\rho(S_m)=\sqrt{m-1}$, we have $\rho(D_{m-2,1})<\rho(S_m^1)\leq \rho(G^*)$. Hence, $\rho(D_{i,m-1-i})< \rho(G^*)$, then $G^*\ncong D_{i,m-1-i}$, and thus $G''\ncong S_{m+1}$.
	
	 Therefore, $G''\in \mG(m,K_{2,r+1})\setminus \{S_{m+1}\}$. It implies a contradiction by $\rho(G'')>\rho(G^*)$. 
\end{proof}

\noindent\textbf{Proof of Theorem \ref{k2,r+1}.}
Let $G^*$ be the extremal graph with the maximal spectral radius in $\mG(m,K_{2,r+1})\setminus \{S_{m+1}\}$, and $X^*=(x_{v_1},\dots,x_{v_{|V(G^*)|}})^T$ be the Perron vector of $G^*$ with $x_{u^*}=\max\limits_{u\in V(G^*)} x_u$. 

Let $A=A(G^*)$, $\rho^*=\rho(G^*)$, $U=N(u^*)$, $W=V(G^*)\setminus N[u^*]$ for short. Then 
\begin{align}
{\rho^*}^2	\xu &=\rho^*(AX^*)_{u^*}= \rho^*\sum_{v\in N(u^*)} x_v=\sum_{v\in N(u^*)} \sum_{w\in N(v)} x_w=\sum_{w\in V(G)} d_{N(u^*)}(w) x_w \notag \\
&=|U|\cdot \xu +\sum_{u\in U} d_U(u)x_u+\sum_{w\in W} d_U(w)x_w. \label{maineq}
\end{align}

It is clear that $S_m^1\in \mG(m,K_{2,r+1})\setminus \{S_{m+1}\}$ and $\rho(S_m^1)$ is the largest real root of the equation $x^3-x^2-(m-1)x+m-3=0$ by Lemma \ref{lemsnk}. Then $\rho(G^*)\geq \rho(S_m^1)>\rho(S_m)=\sqrt{m-1}$. Next we show $G^*\cong S_m^1$. Firstly, we give the following claims by using similar techniques in \cite{zhai2021}.

\noindent\textbf{Claim 1.} For any $u\in U\cup W$, $d_U(u)\leq r$.

Otherwise, $K_{2,r+1}$ is an induced subgraph of $G^*$, a contradiction.

\noindent\textbf{Claim 2.} Let $U_1=\{u\in U \mid \sum\limits_{w\in N_W(u)} d_W(w)>r^2\}$ and $U_2=U\setminus U_1$. If $U_1\neq \varnothing$, then we have $\frac{1}{2}\sum\limits_{u\in U_1} d_U(u)<e(W)$.

Let $N_W(U_1)=\bigcup\limits_{u\in U_1} N_W(u)$. Then $N_W(U_1)\subset W$. Since $d_{U_1}(w)\leq r$ for any $w\in N_W(U_1)$ by Claim 1, $d_W(w)$ is counted at most $r$ times in $\sum\limits_{u\in U_1} \sum\limits_{w\in N_W(u)} d_W(w)$. Thus we have
	\begin{equation}\label{c2a}
		e(W)=\frac{1}{2}\sum_{w\in W} d_W(w)
		\geq \frac{1}{2}\sum_{w\in N_W(U_1)} d_W(w)
		\geq \frac{1}{2r} \sum_{u\in U_1} \sum_{w\in N_W(u)} d_W(w).
	\end{equation}
By the definition of $U_1$ and Claim 1, we have
	\begin{equation}\label{c2b}
\frac{1}{2r} \sum_{u\in U_1} \sum_{w\in N_W(u)} d_W(w)
	> \frac{1}{2r} \sum_{u\in U_1} r^2
	\geq \frac{|U_1|\cdot r}{2}
	\geq \frac{1}{2}\sum\limits_{u\in U_1} d_U(u).
	\end{equation}
		
	Combining (\ref{c2a}) with (\ref{c2b}), Claim 2 holds.

\noindent\textbf{Claim 3.} Let $U_2'= \{u\in U_2 \mid  d_W(u)>\frac{r(\rho^*+2r)}{\rho^*-2r}\}$ and $U_2''=U_2\setminus U_2'$. For any $u\in U_2'$, we have $\frac{1}{2}d_U(u)x_u< \frac{1}{2}d_W(u)\xu-\sum\limits_{w\in N_W(u)} x_w $. 

	Let $u\in U_2'$. Since $U_2'\subset U_2= U\setminus U_1$, we have $\sum\limits _{w\in N_W(u)} \sum\limits _{v\in N_W(w) } x_v \leq \sum\limits_{w\in N_W(u)} d_W(w) \xu \leq r^2\xu$. By Claim 1, $|N_U(w)|\leq r$ for any $w\in N_W(u)$. Thus,
	\begin{equation}\label{c3a}
		\rho^* \sum_{w\in N_W(u)} x_w=\sum_{w\in N_W(u)} \rho^* x_w 
		=\sum_{w\in N_W(u)} \left(\sum_{v\in N_W(w) } x_v+\sum_{v\in N_U(w)}x_v \right)
		\leq (r^2+d_W(u)\cdot r)\xu.
	\end{equation}
	By Claim 1, (\ref{c3a}) and $d_W(u)>\frac{r(\rho^*+2r)}{\rho^*-2r}$, we have
	\begin{equation*}
		\frac{1}{2}d_U(u)x_u +\sum_{w\in N_W(u)} x_w
		\leq \frac{r}{2}\xu+\frac{r}{\rho^*}(r+d_W(u))\xu
		 <\frac{1}{2}d_W(u)\xu.
	\end{equation*}
	Hence, Claim 3 holds.

\noindent\textbf{Claim 4.} For any $u\in U_2''$, $x_u<\frac{1}{2}\xu$.

	Let $u\in U_2''$. By Claim 1 and (\ref{c3a}), we have
	\begin{equation*}
		\rho^*x_u=\xu +\sum_{v\in N_U(u)} x_v+\sum_{w\in N_W(u)}x_w 
		\leq \left(1+r+\frac{r}{\rho^*}(r+d_W(u))\right)\xu.
	\end{equation*}
	By the definition of $U_2''$, $d_W(u)\leq \frac{r(\rho^*+2r)}{\rho^*-2r}$. Thus 
	\begin{equation*}
		\rho^*x_u \leq \left(1+r+\frac{2r^2}{\rho^*-2r}\right)\xu.
	\end{equation*}
	In order to show that $x_u<\frac{1}{2}\xu$, it suffices to show that $1+r+\frac{2r^2}{\rho^*-2r}<\frac{\rho^*}{2}$, or equivalently, ${\rho^*}^2-(4r+2)\rho^*+4r>0$. Since $r\geq 2$ and $\rho^*>\sqrt{m-1}\geq 4r+2 >\frac{4r+2+\sqrt{(4r+2)^2-16r}}{2}$, we have ${\rho^*}^2-(4r+2)\rho^*+4r>0$. Claim 4 follows.

\noindent\textbf{Claim 5.} $W=\varnothing$.

	Note that $\rho^*>\sqrt{m-1}$. By (\ref{maineq}), we have 
	\begin{equation}\label{c5a}
		(m-1)\xu < |U|\cdot \xu +\sum_{u\in U} d_U(u)x_u+\sum_{w\in W} d_U(w)x_w,
	\end{equation}
	where
	\begin{equation}\label{c5b}
		\sum_{u\in U} d_U(u)x_u =\frac{1}{2}\sum_{u\in U_1\cup U_2'} d_U(u)x_u+\left( \frac{1}{2} \sum_{u\in U_1\cup U_2'} d_U(u)x_u+\sum_{u\in U_2''} d_U(u)x_u \right).
	\end{equation}
	Moreover, by Claim 4,
	\begin{equation}\label{c5c}
		\frac{1}{2} \sum_{u\in U_1\cup U_2'} d_U(u)x_u+\sum_{u\in U_2''} d_U(u)x_u\leq \frac{1}{2}\sum_{u\in U}d_U(u)x_{u^*}=e(U)x_{u^*}.
	\end{equation}
	By Claim 2 and Claim 3,
	\begin{equation}\label{c5d}
	\begin{split}
	\frac{1}{2}\sum_{u\in U_1\cup U_2'} d_U(u)x_u
	&\leq \frac{1}{2}\sum\limits_{u\in U_1} d_U(u)x_u+\sum_{u\in U_2'}\left(\frac{1}{2}d_W(u)x_{u^*}-\sum_{w\in N_W(u)} x_w\right)\\
	&\leq e(W)x_{u^*}+\sum_{u\in U}\left(\frac{1}{2}d_W(u)x_{u^*}-\sum_{w\in N_W(u)} x_w\right)\\
	&=e(W)x_{u^*}+\frac{1}{2}e(U,W)x_{u^*}-\sum_{w\in W}d_U(w)x_w.
	\end{split}
	\end{equation}
Combining (\ref{c5a}), (\ref{c5b}), (\ref{c5c}) with (\ref{c5d}), we have
	\begin{equation}
		(m-1)\xu < |U|\cdot \xu +e(U)\xu +e(W)x_{u^*}+\frac{1}{2}e(U,W)x_{u^*}=(m-\frac{1}{2}e(U,W))\xu.
	\end{equation}
	This implies that $e(U,W)<2$. If $e(U,W)=1$, say $E(U,W)=\{uv\}$, then vertex $u$ is a cut vertex. It is a contradiction to Lemma \ref{lemcut}. Therefore, $e(U,W)=0$, and thus $W=\varnothing$ by $G^*$ is connected. Claim 5 follows.

By Claim 5, we have $m=|U|+e(U)$ and ${\rho^*}^2\xu=|U|\cdot \xu +\sum\limits_{u\in U} d_U(u)x_u$ by (\ref{maineq}). 

For any $u\in U$, $\rho^*x_u=\xu+\sum\limits_{v\in N_U(u)}x_v \leq (r+1)\xu$ by Claim 1. Since $\rho^*>\sqrt{m-1}\geq 4r+2>3(r+1)$ and $r\geq 2$, we have $x_u\leq \frac{(r+1)\xu}{\rho^*}<\frac{1}{3}\xu$. Thus by (\ref{c5a}), we have
$$(m-1)\xu < |U|\cdot \xu +\frac{1}{3}\sum_{u\in U} d_U(u)\xu= |U|\cdot \xu +\frac{2}{3}e(U)\xu=(m-\frac{1}{3}e(U))\xu,$$
which implies $e(U)<3$. Note that $G^*$ is not a star. It follows that $e(U)=1$ or $2$.

If $e(U)=2$, then $G^*$ is isomorphic to either $S_{m-1}^2$ or $H$ (see Figure \ref{H}). 

If $G^*\cong S_{m-1}^2$, by Lemma \ref{lemsnk}, $\rho(S_{m-1}^2)$ is the largest root of $f_1(x)=0$ where $f_1(x)=x^3-x^2-(m-2)x+m-6$. Note that
$\rho(S_m^1)$ is the largest root of $f(x)=0$ where $f(x)=x^3-x^2-(m-1)x+m-3$. Since $r\geq 2$, we know $\rho(S_m^1)>\sqrt{m-1}\geq 4r+2\geq 10$. And   $f_1(x)-f(x)=x-3>0$ for $x\geq \rho(S_m^1)$. Thus $\rho(S_{m-1}^2)<\rho(S_m^1)$. It is a contradiction to the definition of $G^*$.

\begin{figure}[h]
	\centering
	\includegraphics[scale=0.13]{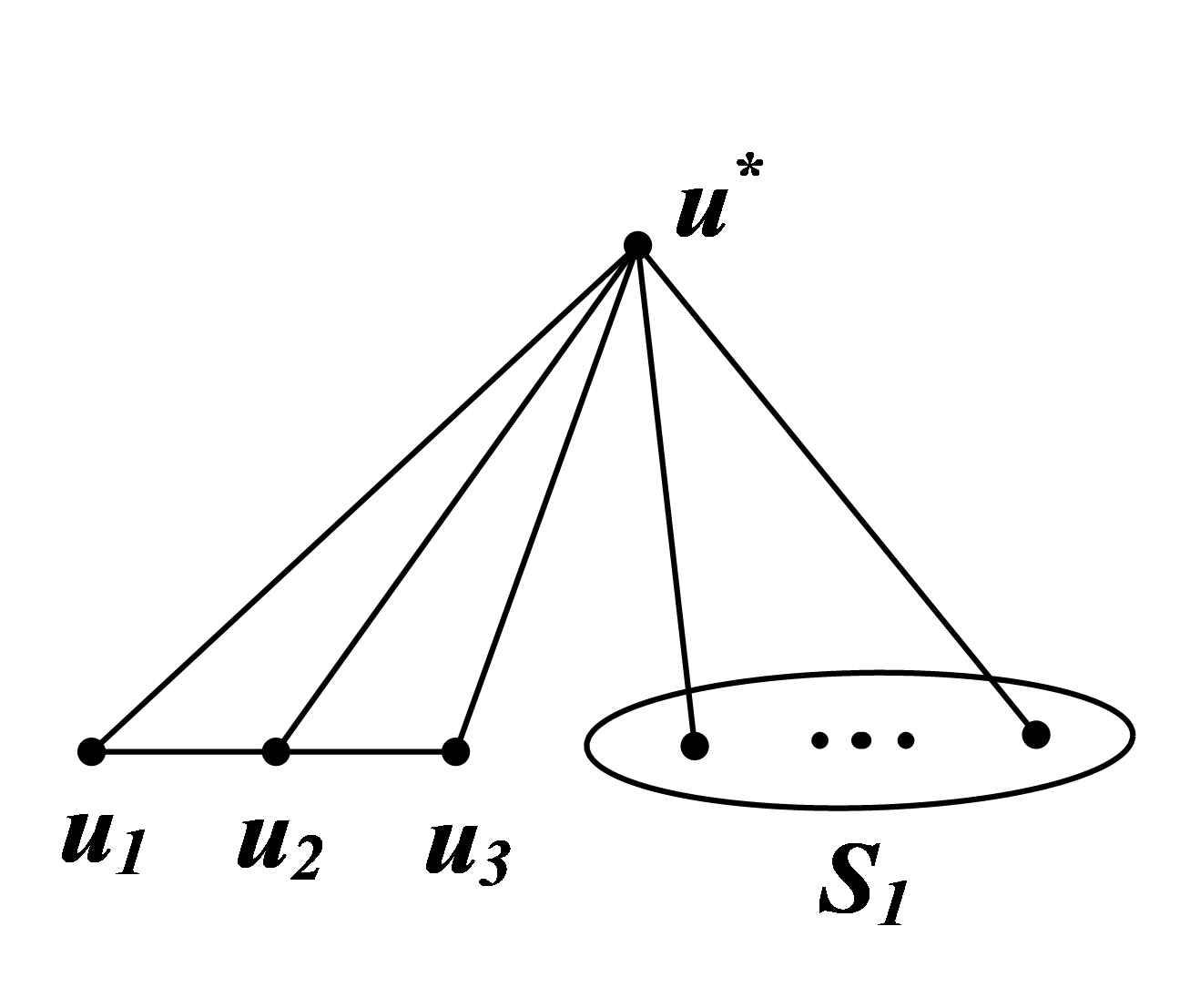}
	\caption{$H$}
	\label{H}
\end{figure}

If $G^*\cong H$, we partition the vertex set of $H$ as $\Pi$: $V(H)=\{u^*\}\cup \{u_1\}\cup \{u_2\}\cup \{u_3\}\cup S_1$, where $|S_1|=m-5$ (see Figure \ref{H}). The quotient matrix of $H$ with respect to $\Pi$ is $$B_{\Pi}=\begin{matrix}
\{u^*\} \\
\{u_1\} \\
\{u_2\}\\
\{u_3\}\\
S_1\end{matrix}\begin{pmatrix}
0 & 1 & 1 & 1 & m-5 \\
1 & 0 & 1 & 0 & 0 \\
1 & 1 & 0 & 1 & 0 \\
1 & 0 & 1 & 0 & 0 \\
1 & 0 & 0 & 0 & 0 \\
\end{pmatrix}.$$
Then $|xI_5-B_{\Pi}|=-x(x^4-mx^2-4x+2m-10)$. Let $f_2(x)=x^4-mx^2-4x+2m-10$, $h_1(x)=f_2(x)-xf(x)=x^3-x^2-(m+1)x+2m-10$. We will show that $h_1(x)>0$ for any $x\geq \rho(S_m^1)$. Since $r\geq 2$, we know $m\geq (4r+2)^2+1> 17$. Then $h_1'(x)=3x^2-2x-(m+1)$, and $h_1'(\sqrt{m-1})=2m-4-2\sqrt{m-1}>0$. This implies that $h_1(x)$ is monotonically increasing when $x>\sqrt{m-1}$. Moreover, $h_1(\sqrt{m-1})=m-9-2\sqrt{m-1}>0$. Hence $h_1(x)=f_2(x)-xf(x)>0$ for $x\geq \rho(S_m^1) >\sqrt{m-1}$. This indicates that the largest root of $f_2(x)$ is less than $\rho(S_m^1)$. By Lemma \ref{lemquo}, we have $\rho(H)<\rho(S_m^1)$. It is a contradiction to the definition of $G^*$.

Therefore, $e(U)=1$ and $|U|=m-1$, then $G^*\cong S_m^1$, and thus $\rho(G)\leq \rho(S_m^1)$ for any $G\in \mG(m,K_{2,r+1})\setminus \{S_{m+1}\}$, with equality if and only if $G\cong S_m^1$. $\hfill \qed$

\section{Proof of Theorem \ref{c5}}
\hspace{1.5em}In this section, we study the extremal graph attaining the maximum spectral radius over all graphs $G\in \mG(m,\theta_{1,2,3})\setminus \{S_{\frac{m+3}{2},2}\}$. 
Similar to Lemma \ref{lemcut}, we have the following lemma.

\begin{lem}\label{lemcut2}
	Let $G^*$ be the extremal graph with the maximal spectral radius in $\mG(m,\ga)\setminus \{ S_{\frac{m+3}{2},2} \}$, and $X^*=(x_{v_1},\dots,x_{v_{|V(G^*)|}})^T$ be the Perron vector of $G^*$ with $x_{u^*}=\max\limits_{u\in V(G^*)} x_u$. Then the following claims hold.\\
	{\rm (i)} $G^*$ is connected.\\
	{\rm (ii)} There exists no cut vertex in $V(G^*)\setminus \{u^*\}$, and hence $d(u)\geq 2$ for any $u\in V(G^*)\setminus N[u^*]$.
\end{lem}
\begin{proof}
	The proof of claim (i) is similar to Lemma \ref{lemcut}(i), so we only prove claim (ii) in the following.
	
	Suppose that there exists a cut vertex in $V(G^*)\setminus \{u^*\}$. Then $G^*$ has at least two end-blocks. We may assume that $B$ is an end-block of $G^*$ with $u^*\notin V(B)$ and a cut vertex $u\in V(B)$. Let $G''=G^*+\{u^*w \mid w\in N(u)\cap V(B) \}-\{uw \mid w\in N(u)\cap V(B) \}$. Then $u^*$ is a cut vertex of $G''$. 
	 Since $\ga$ and $S_{\frac{m+3}{2},2}$ is $2$-connected, we have $G''$ is $\ga$-free and $G''\ncong S_{\frac{m+3}{2},2}$. By Lemma \ref{moveedge}, $\rho(G'')>\rho(G^*)$. It is a contradiction. 
\end{proof}


\begin{lem}\label{rho}
Let $F_{m,t}$ be the graph defined in Section \ref{sec1} with $m>t+1$ {\rm (}see Figure \ref{Fmt}{\rm )}. Then we have\\
	{\rm (i)} 
 $\rho(F_{m,t})< \left\{\begin{matrix}
	\rho(F_{m,1}),	&\text{if } t(\geq 3) \text{is odd,} \\ 
	\rho(F_{m,2}),	&\text{if } t(\geq 4) \text{is even.}
	\end{matrix}\right.$\\
	{\rm (ii)} $\rho(F_{m,t})>\frac{1+\sqrt{4m-7}}{2}$ for $t=1,2$ when $m\geq 22$. 
\end{lem}
\begin{figure}[h]
	\centering
	\begin{minipage}{200pt}
		\centering
		\includegraphics[scale=0.28]{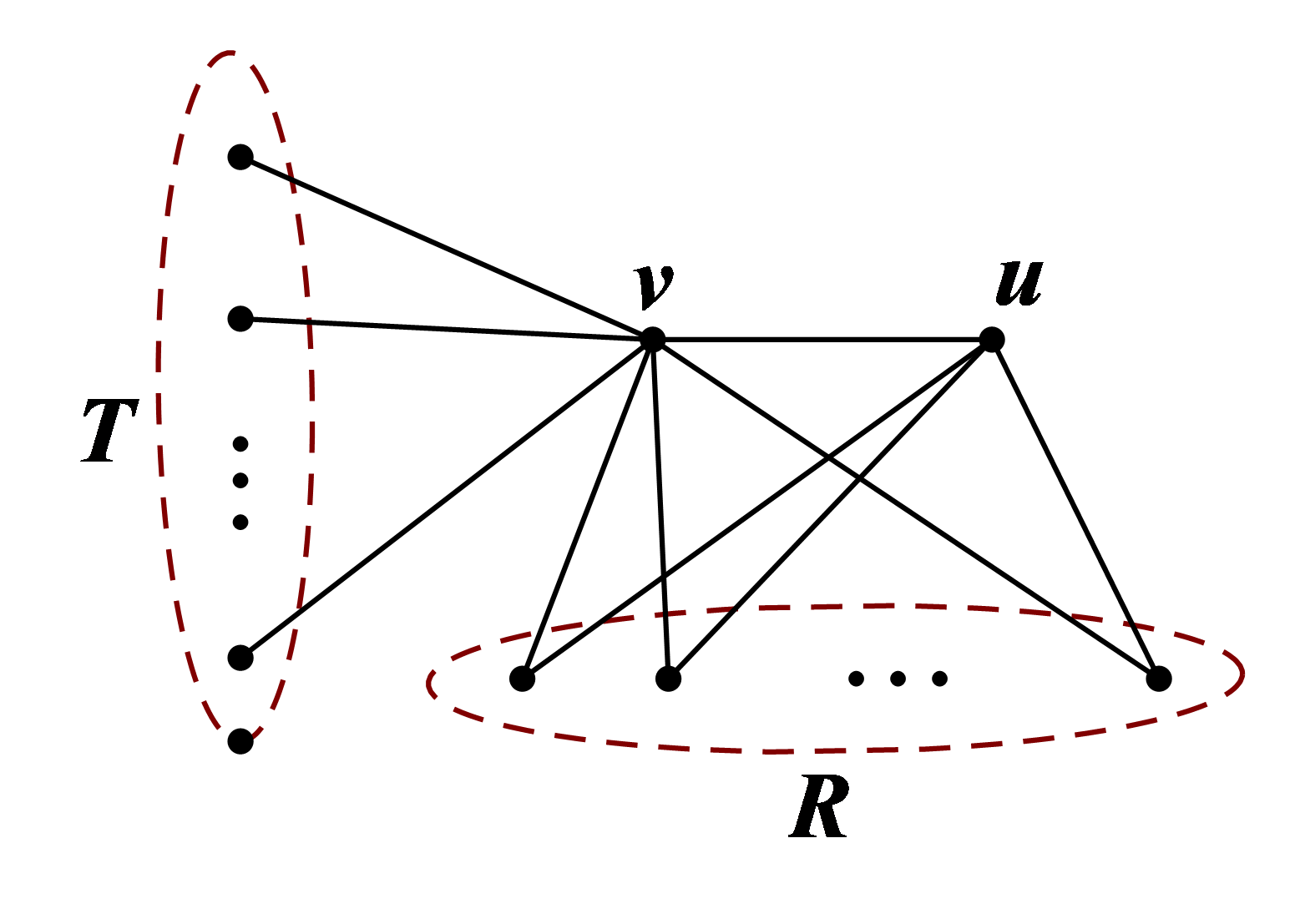}
		\caption{$F_{m,t}$}
		\label{Fmt}
	\end{minipage}
\end{figure}
\begin{proof}
	(i) We partition the vertex set of $F_{m,t}$ as $\Pi_1$: $V(F_{m,t})=T\cup \{v\}\cup \{u\}\cup R$, where $|T|=t$, $|R|=\frac{m-t-1}{2}$ (see Figure \ref{Fmt}). The quotient matrix of $F_{m,t}$ with respect to $\Pi_1$ is $$B_{\Pi_1}=\begin{matrix}
	T\\
	\{v\} \\
	\{u\}\\
	R\end{matrix}\begin{pmatrix}
	0 & 1 & 0 & 0 \\
	t & 0 & 1 & r \\
	0 & 1 & 0 & r \\
	0 & 1 & 1 & 0 
	\end{pmatrix}.$$
	Let $f_3(x)=|xI_4-B_{\Pi_1}|=x^4-mx^2-(m-t-1)x+\frac{t}{2}\cdot (m-t-1)$, $f_4(x)=x^4-mx^2-(m-2)x+\frac{m-2}{2}$. Then $h_2(x)=f_3(x)-f_4(x)=(t-1)x+\frac{t}{4}\cdot (m-t-1)(m-2)>0$ for $x>0$. This indicates that the largest root of $f_3(x)$ is less than the largest root of $f_4(x)$. By Lemma \ref{lemquo}, we have $\rho(F_{m,t})<\rho(F_{m,1})$ for $t\geq 3$. Similarly, we can prove that $\rho(F_{m,t})<\rho(F_{m,2})$ for $t\geq 4$, and thus (i) holds.

	(ii) With the help of Matlab, we obtains that when $m\geq 22$ and $t=1,2$, $f_3(\frac{1+\sqrt{4m-7}}{2})<0$. Thus $\rho(F_{m,t})>\frac{1+\sqrt{4m-7}}{2}$ for $t=1,2$.
\end{proof}

\begin{lem}\label{c5lem1}
	let $G^*$, $u^*$ be the symbols defined in Lemma \ref{lemcut2}. Then the connected component of $G^*[N(u^*)]$ can only be an isolated vertex, a triangle or a star $K_{1,r} (r\geq 1)$.
\end{lem}
\begin{proof}
	If $G^*[N(u^*)]$ contains a path of length $3$, then vertices of the path and $u^*$ induce $\ga$, a contradiction.
\end{proof}

In the rest of this paper, let $G^*$, $X^*$, $u^*$ be the symbols defined in Lemma \ref{lemcut2}, $\rho^*=\rho(G^*)$ and $W=V(G^*)\setminus N[u^*]$. We define $N_i(u^*)=\{ u\in N(u^*)\mid u  \mbox{ lies in the component } K_{i+1} \mbox{ of } G^*[N(u^*)] \}$ and $W_i=N_W(N_i(u^*))=\bigcup\limits _{u\in N_i(u^*)} N_W(u)$ for $i\in \{0,1,2\}$. Since $G^*$ is $\ga$-free, $W_0\cap W_i=\varnothing$ for $i\in\{1,2\}$ and $|N(w)\cap N(u^*)|=1$ for any $w\in  W\setminus (W_0\cup W_1)$. 

Let $c$ be the number of star-components of $G^*[N(u^*)]$, $k$ be the number of triangle-components of $G^*[N(u^*)]$. We have the following results on the local structure of $G^*$.


\begin{lem}\label{c5ew}
	Let $m\geq 22$. Then $e(W)=0$ and $W=W_0$.
\end{lem}
\begin{proof}
	Let $N_+(u^*)=N(u^*)\setminus N_0(u^*)$. Since $\rho^* x_{u^*}=\sum\limits_{u\in N_0(u^*)}x_u+\sum\limits_{u\in N_+(u^*)} x_u$
	and
	${\rho^*}^2	\xu =d(u^*)x_{u^*}+\sum\limits_{u\in N_+(u^*)} d_{N(u^*)}(u)x_u+\sum\limits_{w\in W} d_{N(u^*)}(w)x_w$ by (\ref{maineq}), we have
	\begin{align}
	({\rho^*}^2-\rho^*)\xu & =|N(u^*)|\xu +\sum_{u\in N_+(u^*) }(d_{N(u^*)}(u)-1)x_u-\sum_{u\in N_0(u^*)} x_u+\sum_{w\in W} d_{N(u^*)}(w)x_w \label{c5eq}\\
	&\leq \left(|N(u^*)|+2e(N_+(u^*))-|N_+(u^*)|-\sum_{u\in N_0(u^*)} \frac{x_u}{\xu}+e(W,N(u^*)) \right)\xu.\label{c5eq2}
	\end{align}

	Notice that $F_{m,1}$ is $\ga$-free and $S_{\frac{m}{2}+1,2}$ is a proper subgraph of $F_{m,1}$. By Lemma \ref{snk}, we have $\rho^*\geq \rho(F_{m,1})>\rho(S_{\frac{m}{2}+1,2})=\frac{1+\sqrt{4m-7}}{2}$, that is, ${\rho^*}^2-\rho^*> m-2$. By $m=|N(u^*)|+e(N_+(u^*))+e(W,N(u^*))+e(W)$ and (\ref{c5eq2}), we have 
	\begin{equation}\label{c5eq3}
		e(W)< e(N_+(u^*))-|N_+(u^*)|-\sum_{u\in N_0(u^*)} \frac{x_u}{\xu}+2,
	\end{equation}
	then $e(W)\leq 1$ by $e(N_+(u^*))\leq |N_+(u^*)|$, which follows from Lemma \ref{c5lem1}. 
	
	\noindent\textbf{Case 1.} $e(W)=1$.
	
	Let $w_1w_2$ be the unique edge in $G^*[W]$. It follows from (\ref{c5eq3}) that $c=|N_+(u^*)|-e(N_+(u^*))=0$ and any non-trivial connected component of $G^*[N(u^*)]$ is a triangle. Then $W=W_0\cup W_2$, and thus there are three cases: (1) $w_1,w_2\in W_0$; (2) $w_1\in W_0$, $w_2\in W_2$; (3) $w_1,w_2\in W_2$. In the following, we study the above three cases.
	
	\noindent\textbf{Subcase 1.1.} $w_1,w_2\in W_0$.
	
	Then $e(W_2)=0$. For any $w\in W_2$, we have $|N(w)\cap N(u^*)|=1$, and then $w$ is a cut vertex in $G^*$, in contradiction with (ii) of Lemma \ref{lemcut2}. Thus $W_2=\varnothing$. Assume that the number of triangle-components $k\geq 1$, and let $u_1$ be a vertex of a triangle-component of $G^*[N(u^*)]$. By symmetry, $\rho^*x_{u_1}=2x_{u_1}+\xu$. Since $m\geq 22$, we have $\rho^*>\frac{1+\sqrt{4m-7}}{2}\geq 5$, and then $x_{u_1}=\frac{\xu}{\rho^*-2}<\frac{\xu}{3}$. Then by (\ref{c5eq}), 
	$$({\rho^*}^2-\rho^*)\xu\leq \left(|N(u^*)|+\frac{3k}{\rho^*-2}-\sum_{u\in N_0(u^*)} \frac{x_u}{\xu}+e(W,N(u^*)) \right)\xu.$$
	Since ${\rho^*}^2-\rho^*> m-2$ and $m=e(W)+|N(u^*)|+e(W,N(u^*))+3k$, we have $e(W)<2-2k$. Thus $k=0$ and then $N(u^*)=N_0(u^*)$. By Lemma \ref{lemcut2}, there exist $u_2\neq u_3$ such that $u_2\in N_{N(u^*)}(w_1)$ and $u_3\in N_{N(u^*)}(w_2)$. Since $G^*$ is $\ga$-free, we have $N_{N(u^*)}(w_1)\cap N_{N(u^*)}(w_2)=\varnothing$. Thus $G^*$ is triangle-free. By Theorem \ref{Kr+1}, $\rho^*\leq \sqrt{m}<\frac{1+\sqrt{4m-7}}{2}$, in contradiction with $\rho^*>\frac{1+\sqrt{4m-7}}{2}$.
	
	\noindent\textbf{Subcase 1.2.} $w_1\in W_0$, $w_2\in W_2$.
	
	Then $k\geq 1$ and $N_0(u^*)\neq \varnothing$. Since $G^*$ is $\ga$-free, we have $|N(w_2)\cap N(u^*)|=1$ and $N(w_1)\cap N_2(u^*)=\varnothing$. Then $w_2$ is adjacent to a vertex $u_1$ of a triangle-component $C$ of $G^*[N(u^*)]$. Let $u_2\in V(C)\setminus \{u_1\}$. By symmetry, $\rho^*x_{u_2}=x_{u_2}+x_{u_1}+\xu\leq x_{u_2}+2\xu$, then $x_{u_2}\leq \frac{2\xu}{\rho^*-1}$. By similar proof of Subcase 1.1, we can show $W_2=\{w_2\}$. Then by (\ref{c5eq}) and ${\rho^*}^2-\rho^*> m-2$, we have
	\begin{align*}
		(m-2)\xu &< |N(u^*)|\xu +\frac{3(k-1)\xu}{\rho^*-2} +2x_{u_2}+x_{u_1}-\sum_{u\in N_0(u^*)} x_u+\sum_{w\in W} d_{N(u^*)}(w)x_w\\ 
		&\leq \left(|N(u^*)|+\frac{3(k-1)}{\rho^*-2}+\frac{4}{\rho^*-1}+1-\sum_{u\in N_0(u^*)} \frac{x_u}{\xu}+e(W,N(u^*)) \right)\xu.
	\end{align*}
	
	Hence, $e(W)<3-2k-\sum\limits_{u\in N_0(u^*)}\frac{x_u}{\xu}<1$ by $\rho^*>5$ and $m=e(W)+|N(u^*)|+e(W,N(u^*))+3k$. This contradicts the fact that $e(W)=1$.
	
	\noindent\textbf{Subcase 1.3.} $w_1,w_2\in W_2$.
	
	Then we show $W_2=\{w_1,w_2\}$. Otherwise, there exists $w_3\in W_2$, then $d_{N(u^*)}(w_3)=1$ by $G^*$ is $\ga$-free. Thus $w_3$ is a cut vertex by $E(W,W)=\{w_1w_2\}$, in contradiction with (ii) of Lemma \ref{lemcut2}. It implies that $e(W_2,N(u^*))=2$ by $d_{N(u^*)}(w)=1$ for any $w\in W_2$.
	
	By symmetry, $x_{w_1}=x_{w_2}$, and then $x_{w_1}\leq \frac{\xu}{\rho^*-1}$ by $\rho^*x_{w_1}\leq x_{w_2}+\xu$. By (\ref{c5eq}), ${\rho^*}^2-\rho^*> m-2$ and $m=|N(u^*)|+e(N(u^*))+e(W,N(u^*))+e(W)=|N(u^*)|+e(N(u^*))+\left(e(W_0,N(u^*))+2\right)+1$, we have
	$$(m-2)\xu< |N(u^*)|\xu+e(N(u^*))\xu-\sum_{u\in N_0(u^*)} x_u+e(W_0,N(u^*))\xu+x_{w_1}+x_{w_2}\leq (m-3+\frac{2}{\rho^*-1})\xu.$$
	Hence, $1< \frac{2}{\rho^*-1}$ and then $\rho^*< 4$, a contradiction.
	
	\noindent\textbf{Case 2.} $e(W)=0$.
	
	Suppose that $W\setminus W_0\neq \varnothing$. Let $w\in W\setminus W_0$. Then we show $w\in W_1$. Otherwise, $|N(w)\cap N(u^*)|=1$ by $w\in  W\setminus (W_0\cup W_1)$, and thus $w$ is a cut vertex by $e(W)=0$, in contradiction with (ii) of Lemma \ref{lemcut2}. Therefore, $N(w)=\{u_1,u_2\}$, where $u_1u_2$ is a $K_2$ component of $G^*[N(u^*)]$. By (\ref{c5eq3}), we know that $c=|N_+(u^*)|-e(N_+(u^*)) \leq 1$. Thus for any $w'\in W\setminus W_0$, $N(w')=\{u_1,u_2\}$. Let $G=G^*+\{wu^*\mid w\in W\setminus W_0\}-\{wu_1\mid w\in W\setminus W_0\}$. Clearly, $G$ is $\ga$-free. Moreover, since $\xu\geq x_{u_1}$ and $N_{G^*}(u^*) \subsetneqq N_{G}(u^*)$, we have $\rho(G)>\rho^*$ by Lemma \ref{moveedge}. it is a contradiction. Therefore, $W=W_0$.
	
	Combining the above two cases, we have $e(W)=0$ and $W=W_0$.
\end{proof}

%
%

\begin{lem}\label{ck}
	Let $c$ be the number of star-components of $G^*[N(u^*)]$, $k$ be the number of triangle-components of $G^*[N(u^*)]$. Then $k=0$ if $c=0$, $m\geq 22$ or $c=1$, $m\geq 14$.
\end{lem}
\begin{proof}
	Let $u_1$ be a vertex of a triangle-component of $G^*[N(u^*)]$. By symmetry, $\rho^*x_{u_1}=2x_{u_1}+\xu$. Then
	\begin{equation}\label{rxu}
		x_{u_1}=\frac{\xu}{\rho^*-2}\leq \left\{ \begin{matrix}
			\frac{\xu}{3},& m\geq 22,\\
			\frac{\xu}{2},& m\geq 14.
		\end{matrix}\right.
	\end{equation} by $\rho^*> \frac{1+\sqrt{4m-7}}{2}\geq  \left\{ \begin{matrix}
	    5,& m\geq 22,\\
		4,& m\geq 14.
	\end{matrix}\right.$ 
	
	Let $H$ be the unique star-component of $G^*[N(u^*)]$ when $c=1$. Then by (\ref{rxu}), we have 
	\begin{equation*}\label{xu}
		\sum_{u\in N_+(u^*) }(d_{N(u^*)}(u)-1)x_u\leq  \left\{ \begin{matrix}
			k\xu,& \text{if }c=0, m\geq 22,\\
		\left(e(H)-1+\frac{3}{2}k\right)\xu=\left(e(N_+(u^*)\right)-\frac{3}{2}k-1)\xu	,& \text{if }c=1, m\geq 14.
		\end{matrix}\right.
	\end{equation*}
Combining with (\ref{c5eq}) and $m=|N(u^*)|+e(N(u^*))+e(W,N(u^*))+e(W)$, we have
$$e(W)<  \left\{ \begin{matrix}
	2-2k,& \text{if }c=0, m\geq 22, \\
	1-\frac{3}{2}k,& \text{if }c=1, m\geq 14.
\end{matrix}\right.$$
Thus $k=0$ by $e(W)=0$.
\end{proof}

\begin{figure}[h]
		\centering
		\includegraphics[scale=0.28]{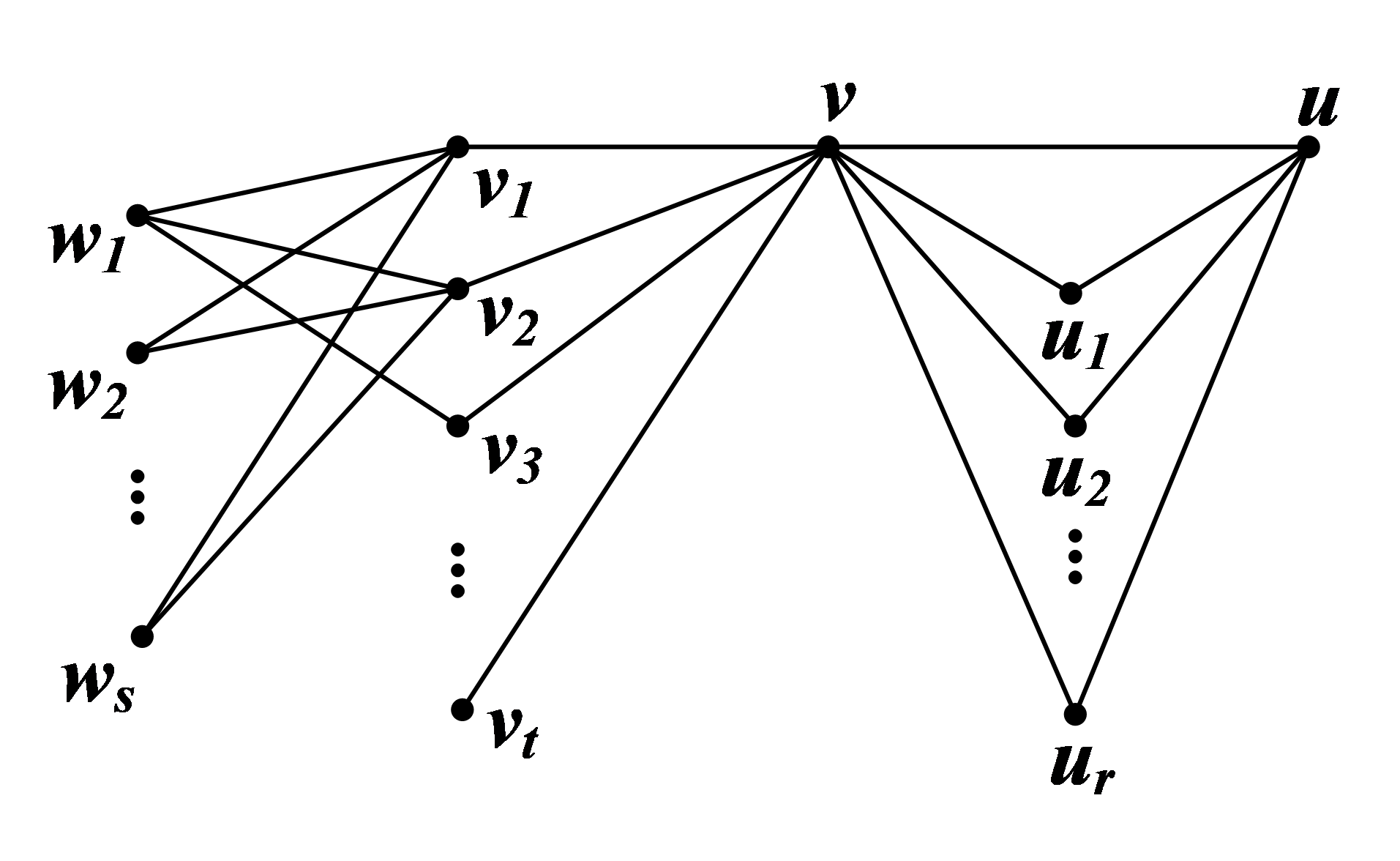}
		\caption{$H_{t,s}\circ K_{1,r}$}
		\label{hts}
\end{figure}

Let $H_{t,s}$ be a bipartite graph with $|T|=t$ and $|S|=s$, $H_{t,s}\circ K_{1,r}$ be the graph obtained by joining a vertex $v$ with all vertices of $K_{1,r}$ and the independent set $T$ of $H_{t,s}$, for $r\geq 1$ and $t,s\geq 0$ (see Figure \ref{hts}).
\begin{lem}\label{-2}
	Let $m\geq 22$. Then there exist integers $r (\geq 1)$, $s(\geq 0)$ and $t(\geq 1)$ such that $G^*\cong H_{t,s}\circ K_{1,r}$, where $H_{t,s}$ is a bipartite graph with size $m-2r-t-1$.
\end{lem}
\begin{proof}
	By Lemma \ref{c5ew}, we have $e(W)=0$ and $W=W_0$. By (\ref{c5eq3}), we have $e(N_+(u^*))> |N_+(u^*)|+\sum\limits_{u\in N_0(u^*)} \frac{x_u}{\xu}-2.$
	Then $c= |N_+(u^*)|-e(N_+(u^*))<2$, say $c=0$ or $1$, and thus $k=0$ by Lemma \ref{ck} and $m\geq 22$.
	
	\noindent \textbf{Case 1.} $c=0$.
	
	Then $G^*$ is triangle-free by $k=0$. By Theorem \ref{Kr+1}, $\rho^*\leq \sqrt{m}<\frac{1+\sqrt{4m-7}}{2}$, a contradiction.
	
	\noindent \textbf{Case 2.} $c=1$.
	
	If $N_0(u^*)=\varnothing$, then $G^*\cong S_{\frac{m+3}{2},2}$. This contradicts with the assumption that $G^*\ncong S_{\frac{m+3}{2},2}$. 
	
	If $N_0(u^*)\neq \varnothing$, without loss of generality, we assume that $H=K_{1,r}$ $( r\geq 1)$ is the unique star-component of $G^*[N(u^*)]$. Then there exist some integers $s(\geq 0)$ and $t(\geq 1)$ such that $G^*\cong H_{t,s}\circ K_{1,r}$, where $H_{t,s}$ is a bipartite graph with size $m-2r-t-1$.
\end{proof}

\begin{lem}\label{s=0}
	If $G^*\cong H_{t,s}\circ K_{1,r}$, then $s=0$.
\end{lem}
\begin{proof}
	By the proof of Lemma \ref{-2}, we know $v=u^*$ (see Figure \ref{hts}). Let $V(K_{1,r})=\{u,u_1,\dots,u_r\}$ with $d_{N(u^*)}(u)=r$. By symmetry, $\rho^* x_{u_1}=\xu+x_u$ and $\rho^*x_u= \xu+rx_{u_1}$. Therefore, $x_u=\frac{\rho^*+r}{{\rho^*}^2-r}\xu$ and $x_{u_1}=\frac{\rho^*+1}{{\rho^*}^2-r}\xu$.
	
	Suppose that $s\neq 0$. Then $W_0\neq \varnothing$ and $N_0(u^*)\neq \varnothing$. If $r\in \{1,2\}$, $G^*-u$ is triangle-free and $\rho(G^*-u)\leq \sqrt{m(G^*-u)}=\sqrt{m-r-1}$ by Theorem \ref{Kr+1}. Then by Lemma \ref{-v}, we have $\rho^*\leq \sqrt{\rho^2(G^*-u)+2r+1}\leq \sqrt{m+r}<\frac{1+\sqrt{4m-7}}{2}$ when $m\geq 22$, a contradiction. Hence, $r\geq 3$.
	
	By (\ref{c5eq}) and ${\rho^*}^2-\rho^*> m-2$, we have
	$$(m-2)\xu< \left(|N(u^*)|+(r-1)\frac{\rho^*+r}{{\rho^*}^2-r}-\sum_{u\in N_0(u^*)} \frac{x_u}{\xu}+e(W,N(u^*))\right)\xu.$$
	Hence, $r-2=e(N_+(u^*))-2< (r-1)\frac{\rho^*+r}{{\rho^*}^2-r}$ by Lemma \ref{c5ew}, that is, $(r-2)({\rho^*}^2-r)-(r-1)(\rho^*+r)<0$.
	
	On the other hand, let $h(x)=(r-2)(x^2-r)-(r-1)(x+r)=(r-2)(x^2-x)-x-2r^2+3r$. Then $h'(x)=2(r-2)x-(r-1)>0$ for $x\geq \frac{1+\sqrt{4m-7}}{2}$, which means $h(x)$ is increasing in the interval $[\frac{1+\sqrt{4m-7}}{2},+\infty)$. 
	Since $d(w)\geq 2$ for each $w\in W_0$ by Lemma \ref{lemcut2}, we have $|N_0(u^*)|\geq 2$ and $m\geq 2r+5$. Notice that $x^2-x=m-2$ if $x=\frac{1+\sqrt{4m-7}}{2}$. Thus,
	\begin{align*}
		h\left(\frac{1+\sqrt{4m-7}}{2}\right)&=(r-2)(m-2)-\frac{1+\sqrt{4m-7}}{2}-2r^2+3r\\
		&=(r-3)(m-2)-2r^2+3r+m-2-\frac{1+\sqrt{4m-7}}{2}\\
		&\geq m-11-\frac{1+\sqrt{4m-7}}{2}>0.
	\end{align*}
	Note that $\rho^*>\frac{1+\sqrt{4m-7}}{2}$. It follows that $h(\rho^*)>0$, a contradiction.
	
	Then $s=0$ and $G^*\cong F_{m,t}$ for some $t\geq 1$.                                       
\end{proof}

\noindent\textbf{Proof of Theorem \ref{c5}.} By Lemma \ref{-2} and Lemma \ref{s=0}, we have $G^*\cong F_{m,t}$ for some $t\geq 1$. We conclude from Lemma \ref{rho} that $G^*\cong F_{m,1}$ if $m$ is even and $G^*\cong F_{m,2}$ if $m$ is odd. The proof is complete.$\hfill \qed$

\end{document}